\documentclass[11pt]{amsart}
\usepackage{graphicx}
\usepackage{amsfonts,amsmath,amssymb,amsthm}
\usepackage{mathtools}
\usepackage[a4paper,margin=1in]{geometry} 
\usepackage{hyperref}
\usepackage{tikz}
\usepackage{comment}
\usetikzlibrary{positioning,calc}

\DeclarePairedDelimiter\norm{\lVert}{\rVert}
\DeclarePairedDelimiter\fracpart{\{}{\}}

\newcommand{\Unit}[1]{\Z_{#1}^\times}
\newcommand{\R}{\mathbb{R}}
\newcommand{\Z}{\mathbb{Z}}
\newcommand{\rbf}{\mathbf{r}}
\newcommand{\vbf}{\mathbf{v}}
\newcommand{\ubf}{\mathbf{u}}
\newcommand{\wbf}{\mathbf{w}}
\newcommand{\ml}{\mathrm{ML}}
\newcommand{\onetok}{(1, 2, \dots, k)}

\newtheorem{theorem}{Theorem}[section]
\newtheorem{lemma}[theorem]{Lemma}
\newtheorem{proposition}[theorem]{Proposition}
\newtheorem{corollary}[theorem]{Corollary}
\newtheorem{conjecture}[theorem]{Conjecture}
\theoremstyle{definition}
\newtheorem{definition}[theorem]{Definition}

\theoremstyle{remark}
\newtheorem{remark}[theorem]{Remark}

\title{Eleven, twelve, and thirteen lonely runners}
\author[T. Sungkawichai]{Touch Sungkawichai}
\date{\today}
\address{London, United Kingdom}
\email{touchsungkawichai@gmail.com}
\author[T. Trakulthongchai]{Tanupat Trakulthongchai}
\address{St John's College, University of Oxford, Oxford, United Kingdom}
\email{tanupat.trakulthongchai@sjc.ox.ac.uk}
\date{\today}
\subjclass[2020]{11K60}

\begin{document}
\begin{abstract}
    Wills conjectured that, for any non-zero integers $u_1,\ldots,u_k$, 
    there is a real number $t$ such that, for all $i=1,\ldots,k$, \[\lVert tu_i\rVert\geq\frac{1}{k+1},\] 
    where $\lVert x\rVert$ is the distance from $x$ to the closest integer. This statement is known as the Lonely Runner Conjecture.
    A computational method developed by Rosenfeld and the second author verified the conjecture for $k\leq9$.
    We further refine this method with new sieving techniques and employ a polynomial method argument to show that
    any $(u_1,\ldots,u_k)\equiv(1,2,\ldots,k)\pmod{p}$ with $\gcd(u_1,\ldots,u_k)=1$ satisfies the conjecture when $k+1$ and $p > k^2+k$ 
    are both odd primes. Ultimately, we provide a computer-assisted proof of the Lonely Runner Conjecture for $k\in\{10,11,12\}$.
\end{abstract}
\maketitle

\section{Introduction}
The Lonely Runner Conjecture, posed by Wills in his thesis \cite{wills1965zwei} in 1965, arose in the study of Diophantine approximation \cite{Wills}. It has been reinterpreted in the context of view obstruction and billiard trajectories 
\cite{cusick1984view}, covering radii of zonotopes \cite{henze_covering_2017}, and the chromatic number of a 
unit-distance graph \cite{zhu2002circular}. 
Writing $\norm{x}$ for the distance from $x$ to the closest integer, namely $\norm{x}=\min(x-\lfloor x\rfloor, \lceil x\rceil-x)$, 
the conjecture can be stated as follows:

\begin{conjecture} \label{lonelyconj}
  Let $u_1, \ldots,u_{k+1}$ be distinct real numbers. For each $i \in\{1,\ldots,k+1\}$, there is a real number $t$ such that
  \[ \norm{t u_i-tu_j} \ge \frac{1}{k+1} \quad \text{for all } j \ne i. \]
\end{conjecture}

Goddyn's \cite{bienia1998flows} popular interpretation gives the conjecture's name: if $k+1$ runners run with distinct,
constant speeds around a unit-length track, each of them will be \emph{lonely} at some time. 
Being lonely means being at least $\tfrac{1}{k+1}$ away from other runners. 
However, the version most commonly worked with fixes one runner to be stationary. 
Additionally, it was shown in \cite{Bohman,henze_covering_2017} (and implicit in \cite{wills1968simultanen,fourrunners}) that only positive integer speeds need to be considered. Hence, Conjecture \ref{lonelyconj} is reduced
to a simpler Conjecture~\ref{lonelyconjred}:

A tuple $(u_1, \dots, u_k)$ is said to have the \emph{LR property} if there is a real number $t$ such that 
\[\norm{tu_i}\geq\frac{1}{k+1} \quad \text{for all } i = 1,\dots,k.\]
Such a time $t$ where the inequalities hold is called a \emph{witness} of the LR property of $(u_1, \dots,u_k)$.

\begin{conjecture} \label{lonelyconjred}
  Every $(u_1,\ldots,u_k)\in\Z_{>0}^k$ has the LR property.

  We denote this statement by $LRC(k)$ or, as customary, the Lonely Runner Conjecture for $k+1$ runners.
\end{conjecture}

Currently, Conjecture \ref{lonelyconjred} is known to hold if $\frac{1}{k+1}$ is replaced with
$\frac{1}{2k}+\frac{1}{k^{5/3+o(1)}}$ \cite{bedert2025riesz}.
More closely related to our approach, it is known for the cases of $k\leq9$ (that is, for at most 10 runners).
The two smallest values of $k$ are trivial. 
Betke~and~Wills~\cite{fourrunners} first proved the case $k=3$, Cusick~and~Pomerance~\cite{cusick1984view} proved $k=4$, 
and Bohman,~Holzman,~and~Kleitman~\cite{Bohman} proved $k=5$. Barajas~and~Serra~\cite{barajas_lonely_2008} 
settled the case $k=6$ in 2008.

This had remained the largest verified case prior to  2025, when Rosenfeld~\cite{Rosenfeld} introduced a new 
framework to handle $k=7$ (and later $k=8$ \cite{rosenfeld2025lonely2}). 
He demonstrated that, with computer assistance, the conjecture can be verified with a tractable finite calculation.
In doing so, he relied on the finite-checking result of Malikiosis, Santos, and Schymura~\cite{Malikiosis}.
The~second~author~\cite{Trakulthongchai}, later that year,
presented a sieving scheme that substantially reduces the work for verification, thereby proving the conjecture for $k\in\{8,9\}$. 

The main objective of this paper is to prove $LRC(k)$ for $k\in\{10,11,12\}$.
\begin{theorem} \label{thm:main}
    $LRC(k)$ holds for $k\leq12$.
\end{theorem}

In order to carry out the proof, we improve upon the framework of \cite{Rosenfeld} and \cite{Trakulthongchai}
both mathematically and computationally. We refine the sieving scheme by adding intermediate sieves and exploit a modulo
$p$ symmetry during implementation to reduce the verification necessary. These were enough to prove $k=11$. 
Furthermore, we employ the polynomial method to establish the following proposition that aids the proof for $k\in\{10,12\}$.
\begin{proposition} \label{prop:onetokintro}
    Let $k+1$ and $p > k^2 + k$ be odd primes and $(u_1,\ldots,u_k)\in\Z_{>0}^k$.
    If $\gcd(u_1,\ldots,u_k)=1$ and 
    \[u_i\equiv i\pmod{p} \quad \text{for all } i = 1,\dots,k,\] 
    then $(u_1,\ldots,u_k)$ has the LR property.
\end{proposition}

Our proof strategy is outlined in the next section. Section \ref{sec:sieves} explains the intermediate sieves 
technique and Section \ref{sec:onetok} is devoted to proving Proposition \ref{prop:onetokintro}. 
We describe the verification algorithms and their implementation in Section \ref{sec:algo}, 
the results of which establish Theorem \ref{thm:main} in Section \ref{sec:result}.
Lastly, we discuss in Section~\ref{sec:conclude} some reasons why the sieves are effective and how the method in this paper may be further extended.

We adopt the following notation. Write $aS=\{as:s\in S\}$ for $a\in\R$ and $S\subseteq\R$. 
The additive group modulo $n$ is denoted $\Z_{n}$, the multiplicative group modulo $n$ is $\Unit{n}$, and 
\[\Z_{p, l}=\{x\in\Z_{lp}: x\not\equiv0\pmod{p}\}.\]
There will be numerous occasions where we treat the standard representatives $\{0,1,\ldots,n-1\}$ in $\Z_n$ as integers, but it should always be clear when this happens.

Let $\pi_{m \to n}:\Z_{m} \to \Z_{n}$ denote the natural projection map when $n\mid m$
and $\pi_n: \Z \to \Z_{n}$ denote the projection modulo $n$ of integers.
The fractional part of $x\in\R$ is denoted by $\fracpart{x}$.

We use boldface letters for vectors, subscripts to indicate coordinates of a vector, 
and $\widehat{\cdot}$ to indicate the omission of a certain coordinate in a vector. 
For example, $u_i$ is the $i^\text{th}$ coordinate of the vector $\ubf$.
We refer to a vector $\ubf$ or $\vbf$ as a \emph{speed tuple} to match the conjecture's settings. 
The application of a scalar function to vectors is coordinate-wise. 

\section{Proof strategy}

For prime $p$ and arbitrary integer $l$, Rosenfeld's \cite{Rosenfeld} key insight is that if $\ubf\equiv\vbf\pmod{lp}$, 
then $\norm{tv_i}=\norm{tu_i}$ for every $i$ whenever $t$ is a rational with denominator $lp$. 
In particular, if we know that $\vbf\in\Z_{lp}^k$ has a witness time of the form $t\in\frac{1}{lp}\Z$,
then every $\ubf\in\pi^{-1}_{lp}(\vbf)$ has the same witness time and thus the LR property.
If every $\vbf\in\Z_{p,l}^k$ has a witness time of the form $t\in\frac{1}{lp}\Z$,
then one of the speeds in any counterexample to $LRC(k)$ must be divisible by $p$. The set of speed tuples $\Z_{p,l}^k$ together with the set of times $\frac{1}{lp}\Z$ is what we informally call an \emph{ansatz}.
The following definition and lemma summarize what must be verified.

\begin{definition}[{\cite[Definition 3]{Trakulthongchai}}\footnote{We change $lp$ in the gcd condition of \cite{Trakulthongchai} to $l$,
  but this changes nothing when $p$ is prime.}]\label{def:improper}
    Let $p$ be prime, and $k \geq 2$ and $l$ be positive integers.
    A speed tuple $\vbf\in\Z_{p,l}^k$ is $(k,p,l)$-\emph{proper} if one of the following conditions holds: 
  \begin{itemize}
    \item there exists $i$ such that $\gcd(l, v_1, \dots, \widehat{v_{i}}, \dots, v_k) > 1$
    \item there exists $t \in\frac{1}{lp}\Z$ such that $\norm{tv_i} \geq \frac{1}{k+1}$ for all $i$.
  \end{itemize}
   If $\vbf$ is not proper, it is \emph{improper}. We denote the set of all improper speed tuples as $I(k,p,l)$.
\end{definition}

\begin{lemma}[{\cite[Lemma 5]{Trakulthongchai}}] \label{lem:divp}
    For $k\geq3$ and prime $p$, if $LRC(k-1)$ is true and $I(k,p,l)$ is empty for some $l$, then any counterexample
    $\ubf$ to $LRC(k)$ satisfies $p \mid u_1 \cdots u_k$. 
\end{lemma}

Lemma~\ref{lem:divp} translates proving Conjecture \ref{lonelyconjred} to the computer verification task
of showing that, for many primes $p$, $I(k,p,l_p)=\emptyset$ for some $l_p$.
Since the choice of $l$ is irrelevant to the final conclusion, 
we introduce a new terminology to simplify our discussion and implementation. 
\begin{definition}
  \label{def:eventualproper}
  Let $k \geq 2$ be an integer and $p$ be a prime. 
  A speed tuple $\vbf \in \Z_{p,1}^k$ is said to be \emph{eventually $(k,p)$-proper}
  if there is a positive integer $l$ such 
  that \[ \vbf \notin \pi_{lp \to p} I(k, p, l). \]
  Let $J(k,p)$ denote the set of all $\vbf \in \Z_{p,1}^k$ that are not eventually $(k,p)$-proper.
\end{definition}

Eventual properness might appear to be a local property of each speed tuple $\vbf$,
but having this local property for every $\vbf$ implies the global conclusion that $I(k,p,l)=\emptyset$ for some $l$.

\begin{lemma}
\label{lem:evenprop}
    For fixed integer $k \ge 2$ and prime $p$, $J(k,p)=\emptyset$ if and only if $I(k,p,l)=\emptyset$ for some positive integer $l$.
\end{lemma}
\begin{proof}
    The backward direction is immediate from the definition. 
    To prove the forward direction, note that $J(k,p)=\emptyset$ implies that,
    for each $\vbf\in \Z_{p,1}^k$, there is $l_\vbf$ for which $\vbf \notin \pi_{l_\vbf p \to p} I(k, p, l_\vbf)$. 
    Then pick \[l=\operatorname{lcm}\{l_\vbf:\vbf\in\Z_{p,1}^k\}.\] 
    
    Fix $\ubf\in\Z_{p,l}^k$. 
    Let $\vbf=\pi_{lp\rightarrow p}\ubf$ and $\vbf'=\pi_{lp\rightarrow l_\vbf p}\ubf$,
    so we can write \[\ubf=(a_1l_\vbf p+v_1',\ldots,a_kl_\vbf p+v_k')\] for some integers $a_1,\ldots,a_k$. 
    
    Notice that $\pi_{l_{\vbf}p \to p} \vbf' = \vbf$. 
    By definition of $l_\vbf$, $\vbf'$ is $(k,p,l_\vbf)$-proper. 
    If $\vbf'$ satisfies the gcd condition,
    then there is $i$ such that $d:=\gcd(l_\vbf, v'_1, \dots, \widehat{v'_{i}}, \dots, v'_k) > 1$. 
    It follows that $d\mid l$ and $d\mid u_j$ for all $j\neq i$, hence $\ubf$ also satisfies the gcd condition.

    Otherwise, there is $t\in\frac{1}{l_\vbf p}\Z$ such that $\norm{tv'_i}\geq\frac{1}{k+1}$ for all $i$. 
    Then observe that $\frac{1}{l_\vbf p}\Z\subseteq\frac{1}{lp}\Z$, so $t\in\frac{1}{lp}\Z$,
    and $\norm{tu_i}=\norm{tv'_i}\geq\frac{1}{k+1}$ for all $i$.

    In either case, $\ubf$ is $(k,p,l)$-proper.
    Therefore, $I(k,p,l)=\emptyset$.
\end{proof}

As a consequence, the emptiness of $J(k,p)$ and $I(k,p,l)$ at some $l$ may be discussed interchangeably.
We may also restate Lemma~\ref{lem:divp} in a more usable form for the proof of our main Theorem \ref{thm:main}.
\begin{corollary}\label{cor:evenpropdivp}
  Let $k \ge 3$ and $p$ be prime.
  If $LRC(k-1)$ is true and $J(k,p)$ is empty, then any counterexample $\ubf$ to $LRC(k)$ has the property that $p \mid u_1\cdots u_k$.
\end{corollary}

Finally, we require an upper bound on the runners' speeds beyond which the LR property is guaranteed.
The best such bound at present was obtained by Malikiosis,~Santos,~and~Schymura~\cite[Theorem A]{Malikiosis}:
\begin{lemma}[{\cite[Corollary 3]{Rosenfeld}}] \label{lem:proddivB}
    For $k \ge 3$, if $LRC(k-1)$ holds, then any counterexample $\ubf$ with $\gcd(u_1,\ldots,u_k)=1$ to $LRC(k)$ must have
    \[ u_1\cdots u_k < B_k \quad \text{where} \quad B_k := \left( \frac{\binom{k+1}{2}^{k-1}}{k} \right)^k. \]
\end{lemma}

Therefore, if we satisfy the precondition of Corollary \ref{cor:evenpropdivp} for a large number of primes,
we can obtain a lower bound on the product of the speeds in any counterexample. 
Together with the stated upper bound, we can conclude the absence of a counterexample.

For convenience, we consolidate the verification task into the following proposition,
which will be used to conclude the main theorem.
\begin{proposition}
\label{prop:verify}
  Let $k \ge 3$ and $P$ be a set of primes such that $J(k,p)=\emptyset$ for all $p\in P$. 
  If $LRC(k-1)$ holds and
  \[\prod_{p \in P} p\geq B_k,\] 
  then $LRC(k)$ is true.
\end{proposition}
\begin{proof}
    Suppose that there exists a counterexample $\ubf$ to $LRC(k)$, which we may assume to have $\gcd(\ubf)=1$. 
    By Corollary \ref{cor:evenpropdivp}, $\prod_{p\in P}p$ divides $u_1\cdots u_k$ and so 
    \[u_1\cdots u_k\geq \prod_{p \in P} p\geq B_k>u_1\cdots u_k,\] 
    where the last inequality comes from Lemma \ref{lem:proddivB}. 
    This is a contradiction.
\end{proof}

\section{Sieving techniques} \label{sec:sieves}
In order to carry out our proof strategy, we need to verify that $J(k,p)=\emptyset$. 
The first approach is to check all tuples in $\Z^{k}_{lp}$ to verify that $I(k,p,l)=\emptyset$,
but this entails checking approximately $(lp)^k$ tuples. 
Moreover, for reasons to be given in Remark~\ref{rem:tightset},
we need $l\geq k+1$, so this approach becomes very laborious even for $k=8$ \cite{rosenfeld2025lonely2}.

Before we motivate them, let us at once formalize and validate  in the following proposition the two techniques that allow for efficient verification. 

\begin{proposition}
\label{prop:liftproj} 
    Fix an integer $k \ge 2$, prime $p$, and assume that $\vbf \in J(k, p)$.
    Let $l_1 = 1$ and $S_1 \subseteq \Z^k_{p}$ contain $\vbf$. 
    For $n \ge 2$, let $S_n$ be obtained from $S_{n-1}$ in one of the following ways:
    \begin{itemize}
        \item (lifting to $l_n$) Let $l_n$ be a multiple of $l_{n-1}$ and set
        \[ S_n = \{\ubf\in\pi_{l_n p\rightarrow l_{n-1}p}^{-1}S_{n-1}:\ubf\text{ is $(k,p,l_n)$-improper}\}.\] 
        \item (projecting) Let $l_n = 1$ and set
        \[ S_{n}=\pi_{l_{n-1}p\to p}S_{n-1}. \]
    \end{itemize}
    Then, $S_n \ne \emptyset$ for all $n$.
    Moreover, if $l_n = 1$, then $\vbf \in S_n$.
\end{proposition}
\begin{proof}
    Fix $k,p$ and eventually improper $\vbf$. We define the improper fiber of $\vbf$ at level $l$ by
    \[ F_{l} = \pi_{lp \to p}^{-1} (\vbf) \cap I(k, p, l). \]
    Note that $F_l\neq\emptyset$ for every $l$ by Definition \ref{def:eventualproper}.
    
    Let $(l_n)$ and $(S_n)$ be the sequences described in the proposition. We will show by induction that $F_{l_n} \subseteq S_n$ for every $n$. This is obvious for $n=1$ where $F_1=\{\vbf\}$.
    
    Now, assume that $F_{l_{n-1}} \subseteq S_{n-1}$. First suppose that $S_n$ is obtained by lifting to $l_n$, so \[ S_n = \pi_{l_np \to l_{n-1}p}^{-1} S_{n-1} \cap I(k, p, l_n). \]
    Let $\ubf \in F_{l_n}$, then $\ubf \in I(k, p, l_n)$ and 
    $\pi_{l_np \to p} (\ubf) = \vbf$. To show that $\ubf \in S_n$, 
    it remains to show that $\pi_{l_np \to l_{n-1}p}(\ubf) \in S_{n-1}$. Write $\wbf = \pi_{l_np \to l_{n-1}p} (\ubf)$. Then, we know that $\pi_{l_{n-1}p \to p}(\wbf) = \pi_{l_np \to p}(\ubf) = \vbf$ since $l_n$ is divisible by $l_{n-1}$. By \cite[Lemma 7]{Trakulthongchai}, we also have that $\wbf \in I(k, p, l_{n-1})$.
    Thus, $\wbf \in F_{l_{n-1}} \subseteq S_{n-1}$ by the induction hypothesis. 

    On the other hand, suppose that $S_n$ is obtained by projection (thus $l_n = 1$). Recall that $F_{l_{n-1}}$ is non-empty, so there exists $\wbf \in F_{l_{n-1}}$. Hence $\wbf \in S_{n-1}$ by the induction hypothesis and $\pi_{l_{n-1}p \to p}(\wbf) = \vbf$. Therefore, $\vbf \in \pi_{l_{n-1}p \to p} (S_{n-1}) = S_n $.
    Since $F_{l_n} = F_{1} = \{\vbf\}$, the induction is complete.

    For all $n$, we have $S_n\supseteq F_{l_n}\neq\emptyset$, so $S_n$ is non-empty. In addition, for $n>1$ we have $l_n = 1$ only after projecting, and in that case we have shown that $S_n\supseteq F_{l_n}=\{\vbf\}$.
\end{proof}

The two actions are lifting, which we generalize from \cite{Trakulthongchai}, and backward projection, a new technique.
Proposition~\ref{prop:liftproj} ensures that both actions preserve the improper fiber of every $\vbf \in J(k,p)$, 
which means that we cannot lose a non-eventually proper speed tuple in the process. 
Now, we will motivate each of them by their applications.

\subsection{Lifting}
\label{subsec:efflift}
Lifting is the main innovation of \cite{Trakulthongchai}. 
The process is described in the first construction of $S_n$ from $S_{n-1}$ in Proposition~\ref{prop:liftproj}: suppose we have a tuple in $\Z^k_{p,l}$,
we can \emph{lift} it to the fiber in $\Z^k_{p,cl}$, and then discard all proper speed tuples.

\begin{remark}
\label{rem:tightset}
    Let $k\geq3$ and $p>k+1$. Note that the tuple $\onetok$ is lonely precisely when $t=\frac{s}{k+1}$ with $\gcd(s,k+1)=1$. 
    Hence, if $k+1\nmid l$, there is no such time in $\frac{1}{lp}\Z$,
    so the tuple $\onetok$ is $(k,p,l)$-improper. 
    In other words, $I(k,p,l)=\emptyset$ only if $k+1\mid l$. 
    As a consequence, if $k+1$ is prime and we start with tuples in $\Z^k_{p}$
    we need to lift with $k+1$ (or a multiple of $k+1$, but this would be suboptimal)
    at least once because we can eliminate $\onetok$ only once we get to $\Z^{k}_{p(k+1)}$.
\end{remark}

Even though it is impossible that $I(k,p,1)=\emptyset$, in view of applying Proposition~\ref{prop:liftproj}, 
we still compute $I(k,p,1)$ first.
This is because $I(k, p, 1)$ is a non-trivial \emph{upper bound}, a superset, of $J(k, p)$ 
that is very easy to compute as it requires checking only approximately $p^k$ speed tuples.
After obtaining that, we can lift with a divisor of $k+1$ to obtain a better upper bound.

Observe that lifting a set $S$ of tuples with a multiplier $c$ requires checking at most
\[ \left| \pi^{-1}_{clp \to lp} S \right| = c^k\cdot|S| \]
speed tuples. Empirically, $|S|$ is found to be smaller as we lift\footnote{
  Even though lifting increases the number of speed tuples to be checked,
  it seems that the increased precision of the allowed time in the lifted ansatz has a greater effect \cite{Trakulthongchai}.
}, so we take $c$ to be small at first. 
Furthermore, as $k$ gets larger, the value of $c$ has a greater effect on the computational time, 
and the $c=2$ lift becomes comparatively cheap.

The key improvement from \cite{Trakulthongchai} is noting that, in fact, we can take this idea to the extreme.
Since the $c=2$ lift is much easier to compute, 
we can apply a $c=2$ lift first (and successively) whether or not $2$ is a prime factor of $k+1$.
This trims down the number of speed tuples in $S$ before we apply the necessary lifts with larger multiplier $c>2$.
Even better, doing this will never increase the number of speed tuples at the end because we can always project backward,
as we will explain next.
 
\subsection{Backward projection}
\label{subsec:backproj}
Backward projection is just the application of $\pi_p$ to pull back the speed tuples to $\Z^k_{p}$.
Just as lifting can reduce the size of the search space, projection achieves a similar effect since, 
for any set $S\subseteq\Z_{p,l}^k$, we have
\[|\pi_p S| \leq |S|.\]
Projecting becomes counterproductive only if we lift with a constant $c$ that was previously used,
because then we are intersecting with the same improper set and no speed tuples can be further eliminated.

The strategy that we apply in practice (see Subsection~\ref{subsec:algo}), for $k+1$ prime,
is to lift with $c=2$ successively and project back.
The few remaining tuples are handled separately, as will be discussed in Section~\ref{sec:onetok}.
For $k+1$ composite, if we project before $k+1$ is reached, 
we are bound to miss eliminating some speed tuples. 
Although this can be remedied, to streamline the proof we choose not to do projection for $k=11$. 

We remark that this strategy can be extended to the following:
lift with $c=2$ successively and project back, then lift with $c=3$ successively and project back, and so on;
this extension is unnecessary for our present goal.

\section{Eventual properness of
  \texorpdfstring{$(1, 2,\ldots, k)$}{(1, 2, ..., k)}
} \label{sec:onetok}
According to Remark~\ref{rem:tightset}, when $k+1$ is prime, we have to lift with the multiplier $c=k+1$. 
This lift is extremely computationally intensive: for $k=12$, the $c=k+1$ lift requires checking $(13/2)^{12}\approx5\cdot10^9$ 
times more tuples than the $c=2$ lift (for fixed $p$ and input set).
If we are to verify the conjecture for $k\in\{10,12\}$ in a reasonable time, we need to avoid doing this lift computationally.
For many primes $p$, after lifting $I(10,p,1)$ with $c=2$ several times and projecting back,
the only tuples that cannot be shown eventually proper are, up to symmetry, $( 1, 2, \ldots, 10 )$. 
The analogous observation holds when $k=12$.

Therefore, the aim of this section is to determine the eventual properness
of $\onetok$. 
We emphasize that our method still requires \emph{lifting} with $c=k+1$. 
In fact, we are merely finding a solution for speed tuples in $\pi^{-1}_{(k+1)p\rightarrow p}\onetok$ in the $(k,p,k+1)$-ansatz,
but we do so analytically instead of computationally.

Proposition \ref{prop:polynomial} is the main combinatorial input for this task. It is also of interest in its own right as a fact of discrete geometry, when we interpret the construction as a discrete plane.
For convenience, we write 
\[ N_{k} := \{ \vbf \in \Z_{k+1}^k : \vbf \ne \mathbf0 \text{ and } \vbf \text{ has at least one zero coordinate}\}. \]
\begin{proposition}
\label{prop:polynomial}
    Let $k+1$ be an odd prime. For all vectors $\vbf \in N_k$,
    there are $s$ and $r$ in $\Unit{k+1}$ such that
    \[s\vbf + r\onetok  \in  \{1,\ldots,k-1\}^k\] as a vector in $\mathbb{Z}_{k+1}^k$.
\end{proposition}
\begin{proof}
    In this proof, we work in the field $\Z_{k+1}$.
    Define \[V_m = \{v_i+mi:i\in \{1,\ldots,k\} \}\] for $m \in \Unit{k+1}$.
    We say $m\in \Unit{k+1}$ is \emph{good} if $0 \notin V_m$ and \emph{bad} otherwise. 
    Let $G$ be the set of good $m$'s.

    Now, as $\vbf$ contains at least one zero coordinate, it contains at most $k-1$ non-zero coordinates.
    For each non-zero coordinate $v_i$, there is exactly one choice of $m\in\Unit{k+1}$ such that $v_i + m i = 0$, 
    whereas for each zero coordinate, there cannot be such $m$. 
    Thus, there are at most $k-1$ bad units.
    This implies that $|G| \ge 1$.

    We continue the proof by contrapositive. Assume that for all $r$ and $s$, 
    \[s\vbf + r\onetok  \notin  \{1,\ldots,k-1\}^k,\] then we will show that $\vbf = 0$.

    Let $m$ be good and $s\in\Unit{k+1}$. Then $r:=sm\in\Unit{k+1}$ and so, by assumption,
    \[ s\vbf + r\onetok \not\in \{1, \dots, k-1\}^k. \]
    Thus, there is an index $i$ such that $s (v_i + m i) \in \{0, -1\}$.
    As $m$ is good and $s\neq0$, it is impossible for the product to be $0$, and so it must be $-1$. Hence, 
    \[ v_i + m i = -s^{-1}. \]
    As $s$ was chosen arbitrarily, $V_m = \Unit{k+1}$.

    From $V_m = \Unit{k+1}$ when $m$ is good, we have that
    \[k! = \prod_{i=1}^k(v_i+mi) = \prod_{i=1}^k i \cdot \prod_{i=1}^k (v_i i^{-1} + m) = k!\prod_{i=1}^k (v_i i^{-1}+m),\]
    which gives \[\prod_{i=1}^k (v_i i^{-1} + m) = 1.\]
    On the other hand, when $m$ is bad, there is an index $i$ such that $v_i + m i = 0$, hence 
    \[ \prod_{i=1}^k (v_i i^{-1} + m) = \prod_{i=1}^k (v_i + m i) i^{-1} = 0. \]
    Therefore, the polynomial \[ P(X) := \prod_{i=1}^k (v_ii^{-1}+X) \] is the indicator function of $G$ on $\Unit{k+1}$.
    Note also that since $\vbf$ contains at least one zero coordinate, $P(0) = 0$.

    As $k+1$ is prime, Fermat's Little Theorem gives another indicator function of $G$, which is
    \[ Q(X) := \sum_{m \in G} (1-(X-m)^k), \] 
    where the sum is not empty as $|G| \ge 1$.
    Thus, $Q$ is a polynomial of degree $k$.
    
    As $P(0) = 0 = Q(0)$, $P$ and $Q$ are polynomials of degree at most $k$ that agree on $k+1$ points, hence $P = Q$.
    In particular, their leading coefficients must be equal, that is 
    $1 \equiv -|G| \pmod{k+1}$. Since $|G| \le k$ as integers, we must have that $|G| = k$.
    Therefore, there is no bad $m$. Specifically, for any $i\in\{1,\ldots,k\}$, $v_i+mi \neq 0$
    for all $m\in\Unit{k+1}$. 
    However, this implies $v_i = 0$ for all $i$, that is, $\vbf = \mathbf{0}$.
\end{proof}

The strategy we will later use to prove Proposition~\ref{prop:corlarge} is decoupling and limiting dependency on the prime $p$. 
Instead of considering times of the form $t=\frac{a}{(k+1)p}$, we decompose them into partial fractions $t=\frac{s}{k+1}+\frac{r}{p}$.
The former term has no dependence on $p$. To capture the effect of $p$ of the latter term, we define 
\[ \rbf_k(t):=\left( \lfloor (k+1) \fracpart{t} \rfloor, \; \lfloor (k+1)\fracpart{2t} \rfloor, \; \ldots, \; \lfloor  (k+1) \fracpart{kt} \rfloor \right) \] 
for $t\in\R$.
Intuitively, $\lfloor (k+1) \fracpart{it} \rfloor = x$ means that the runner with speed $i$ is in the arc 
$\left[ \frac{x}{k+1}, \frac{x+1}{k+1} \right)$ of the unit circle at time $t$.
With this discretization, we state Lemma~\ref{lem:resolvekplusone}.

\begin{lemma}
\label{lem:resolvekplusone}
    Let $k+1$ be an odd prime. For every $\vbf \in N_k$, there are $s \in \Z_{k+1}$ and $r \in \Z$ for which 
    \[ s\vbf + \rbf_k(\tfrac r{k+1}) \in \{1,\ldots,k-1\}^k\] when considered modulo $k+1$.
\end{lemma}
\begin{proof}
  Let $\vbf$ be any speed tuple in $N_k$.
  By Proposition~\ref{prop:polynomial},
  there exist $s,r\in\Unit{k+1}\subseteq\Z_{k+1}$ such that
  \[s\vbf + r\onetok  \in \{1,\ldots,k-1\}^k.\] 
  By definition, \[\begin{split}
      \rbf_k\left(\tfrac{r}{k+1}\right) 
      &=\left(\left\lfloor (k+1)\fracpart*{\tfrac{r}{k+1}} \right\rfloor,\left\lfloor (k+1)\fracpart*{\tfrac{2r}{k+1}} \right\rfloor,\ldots, 
        \left\lfloor (k+1)\fracpart*{\tfrac{kr}{k+1}}\right\rfloor \right)\\
      &=(r,2r,\ldots,kr)\\&=r\onetok
  \end{split}\]
  modulo $k+1$. Therefore,
  \[s\vbf+\rbf_k \left( \tfrac{r}{k+1} \right) \in \{1,\ldots,k-1\}^k,\]
  modulo $k+1$.
\end{proof}

A corollary of Lemma~\ref{lem:resolvekplusone} is that $\onetok$ is eventually $(k,k+1)$-proper, but this is trivial.
In fact, we only need that \[s\vbf + r\onetok  \in  \{1,\ldots,k\}^k\] 
to show eventual $(k, k+1)$-properness of $\onetok$.
However, proving the stronger statement that 
\[s\vbf + r\onetok  \in  \{1,\ldots,k-1\}^k\]
in Proposition~\ref{prop:polynomial} allows us to extend the result to larger primes via the next lemma.

\begin{lemma}
\label{lem:resolvep}
    Let $k+1$ and $p > k(k+1)$ be odd primes. For every $\vbf \in N_k$,
    there are $s \in \Z_{k+1}$ and $r \in \Z$ for which 
    \[ s\vbf + \rbf_k(r/p) \in \{1,\ldots,k-1\}^k\] 
    when considered modulo $k+1$.
\end{lemma}
\begin{proof}
  From Lemma~\ref{lem:resolvekplusone}, there exist $s \in \Z_{k+1}$ and $r \in \Z$ for which
    \[ s\vbf + \rbf_k(r/(k+1)) \in \{1,\ldots,k-1\}^k.\]
  Thus, it suffices to show that $\rbf_k(\tfrac{1}{k+1} \Z) \subseteq \rbf_k(\tfrac{1}{p} \Z)$ for all $p>k(k+1)$.

  Fix $p$ to be a prime greater than $k(k+1)$. Notice that $\rbf_k$ is a piecewise constant function and that 
  $x$ is a discontinuity of $\rbf_k$ if and only if it is a discontinuity of $\left\lfloor (k+1) \fracpart{ix} \right\rfloor$ for some $i$.
  Those points of discontinuity are precisely points of the form  $\frac{a}{i(k+1)}$ for integer $a$ and $1 \le i \le k$.

  Choose an integer $n$ and denote $x := \tfrac{n}{k+1}$. Observe that $x$ is a discontinuity of $\rbf_k$ and let $y$ be 
  the smallest number greater than $x$ that is also a discontinuity of $\rbf_k$. We know that $y$ can be written as 
  $\tfrac{a}{i(k+1)}$ for some integer $a$ and $1 \le i \le k$, so it follows that 
  \[ y - x = \frac{a}{i(k+1)} - \frac{n}{k+1} = \frac{a - ni}{i(k+1)} \ge \frac{1}{i(k+1)} \ge \frac{1}{k(k+1)}. \]

  As $y$ was chosen to be the smallest such discontinuity, $\rbf_k$ is constant on the interval $[x, y)$. 
  Since the interval's width is at least $\tfrac{1}{k(k+1)} > \tfrac{1}{p}$, 
  there must be an integer $m$ such that $\tfrac{m}{p} \in [x, y)$.
  For such a value of $m$, $\rbf_k (\frac{n}{k+1}) = \rbf_k (\frac{m}{p})$.
  As $n$ is arbitrary, the desired inclusion follows.
\end{proof}

With this lemma, we finally obtain the headline result that assures eventual properness of $\onetok$ for $p > k(k+1)$.

\begin{proposition}
    \label{prop:corlarge}
    Suppose that $k+1$ and $p>k(k+1)$ are odd primes.
    Then, every speed tuple in $\pi_{(k+1)p\rightarrow p}^{-1}\onetok$ is $(k,p,k+1)$-proper, 
    so in particular $\onetok$ is eventually $(k,p)$-proper.
\end{proposition}
\begin{proof}
    We will show that every $\ubf \in \pi_{p(k+1) \to p}^{-1}\onetok$ is $(k,p,k+1)$-proper, 
    and hence $\onetok$ is eventually $(k,p)$-proper.
    Let $u_i = a_ip+i$ for some $a_i\in\Z_{k+1}$ for $i=1,\ldots,k$.
    By considering a time of the form $t=\frac{s}{k+1}+\frac{r}{p}$ for $s\in\Z$ and $r\in\Z$, 
    note that the fractional part of $tu_i$ is given by
    \[ \fracpart{tu_i}=\fracpart*{\left(\frac{s}{k+1}+\frac{r}{p}\right)(a_ip+i)}=\fracpart*{\frac{su'_i}{k+1}+\frac{ri}{p}} \] 
    where $\ubf' := \pi_{p(k+1)\to k+1} \ubf$. 
    
    If $\ubf'=\mathbf0$, then it follows that $k+1\mid\gcd(\ubf)$, thus \[k+1\mid\gcd(k+1,\ubf)\]
    hence $\ubf$ is $(k,p,k+1)$-proper because it satisfies the gcd condition.

    If $\ubf'$ contains no zero coordinate, then pick $r=0$ and $s=1$ and we have 
    \[ \fracpart{tu_i}=\fracpart*{\frac{u'_i}{k+1}}. \]
    Since $u_i' \ne 0$, $\norm{tu_i} \ge \frac{1}{k+1}$ for all $i$.

    Otherwise, $\ubf'\in N_k$, so by Lemma~\ref{lem:resolvep} there is a choice of $r$ and $s$ such that 
    \[ s\ubf' + \rbf_k(r/p) \in \{1,\ldots,k-1\}^k \]
    when considered modulo $k+1$. 
    This means that for each coordinate $i$, there are $A_i$ and $B_i$ such that 
    $\fracpart{\frac{su'_i}{k+1}}=\frac{A_i}{k+1}$ and 
    $\fracpart{\frac{ri}{p}}\in\left[\frac{B_i}{k+1},\frac{B_i+1}{k+1}\right)$ 
    where $C_i:=A_i+B_i\pmod{k+1}$ is in $\{1,\ldots,k-1\}$ when considered modulo $k+1$. 
    Therefore,
    \[\fracpart{tu_i}=\fracpart*{\frac{su'_i}{k+1}+\frac{ri}{p}}\in\left[\frac{C_i}{k+1},\frac{C_i+1}{k+1}\right)\]
    and in particular the set on the right-hand side does not intersect with $[0,\frac{1}{k+1}) \cup (\frac{k}{k+1},1)$.  
    Since this holds for every $i$, $\ubf$ has the LR property.

    As there exists a witness time $t=\frac{s}{k+1}+\frac{r}{p}\in\frac{1}{p(k+1)}\Z$, $\ubf$ is $(k,p,k+1)$-proper.
\end{proof}

Proposition \ref{prop:corlarge} implies Proposition \ref{prop:onetokintro}: 
suppose that $\gcd(\vbf)=1$ and $\pi_p(\vbf)=\onetok$, then we have shown in Proposition~\ref{prop:corlarge}
that $\pi_{(k+1)p}(\vbf)$ is $(k,p,k+1)$-proper. 
If it is proper due to the gcd condition, then as $\gcd(\vbf)=1$, by the pre-jump technique (see e.g. \cite[Lemma 3]{Rosenfeld}), 
$\vbf$ has the LR property. 
Otherwise, there is a time $t\in\frac{1}{(k+1)p}\Z$ such that $\norm{tv_i}\geq\frac{1}{k+1}$ for all $i$.

\section{Implementation}
\label{sec:algo}

\subsection{Equivalence of speed tuples}
\label{subsec:equiv}
Some speed tuples in an ansatz are essentially the same, so we only need to verify properness for one of them.
We say that two speed tuples $\ubf$ and $\vbf$ in $\Z_{p,1}^k$ are \emph{equivalent} if $\vbf$ can be obtained from $\ubf$ 
by performing a finite number of the following actions:
\begin{itemize}
    \item Permuting the coordinates of $\ubf$
    \item Flipping the sign of a coordinate $u_i$ to $-u_i$
    \item Multiplying $\ubf$ by some $a\in\Unit{p}$.
\end{itemize}

The first two actions are used in both \cite{Rosenfeld} and \cite{Trakulthongchai}.
For completeness, we show that this equivalence indeed preserves eventual properness.

\begin{proposition}
  \label{prop:equivproper}
  Let $p$ be prime and $\ubf,\vbf\in\Z_{p,1}^k$ be equivalent.
  If $\ubf$ is eventually $(k,p)$-proper, then so is $\vbf$.
\end{proposition}
\begin{proof}
  We only need to show that each of the aforementioned actions preserves eventual properness.
  This is trivial for permuting.
  Suppose $\ubf$ is eventually proper and $l$ is the integer for which every element in $\pi^{-1}(\ubf)$ is proper,
  where $\pi:=\pi_{lp\rightarrow p}$. Now let $\vbf'\in \pi^{-1}(\vbf)$.

  If $\vbf=(u_1,\ldots,-u_i,\ldots,u_k)$, 
  then $\vbf'=(a_1p+u_1,\ldots,a_ip-u_i,\ldots,a_kp+u_k)$ for some integers $a_1,\ldots,a_k$. 
  In this case, pick 
  \[\ubf':=(a_1p+u_1,\ldots,(l-a_i)p+u_i,\ldots,a_kp+u_k)\in\pi^{-1}(\ubf).\]

  If $\vbf=a\ubf$ for some $a\in\Unit{p}$, 
  then by Chinese Remainder Theorem, writing $lp=p^em$ where $p\nmid m$, there is $b\in\Z_{lp}$ such that $b\equiv a^{-1} \pmod{p^e}$ and $b\equiv1\pmod{m}$. 
  In particular, $b\in\Unit{lp}$ and 
  \[\ubf':=b\vbf'\in\pi^{-1}(\ubf)\]
  because $\pi(b\vbf')=\pi(b)\pi(\vbf')=a^{-1}(a\ubf)=\ubf$. 
  % If a time $t$ is a witness for $\ubf$, then the time $bt$ is a witness for $\vbf$.

  Then, whichever way $\ubf'$ satisfies properness,
  it can be checked that $\vbf'$ also satisfies properness in the same way.
  Therefore every speed tuple in $\pi^{-1}(\vbf)$ is proper, and thus $\vbf$ is eventually proper.
\end{proof}

Therefore, we may only verify the eventual properness for a single element in each equivalence class. 
We implement this by ordering the speeds ascendingly,
limiting the speeds to be in the first half of residues modulo $p$, and fixing $v_1$ to be $1$.
In other words, we check all tuples of integers $(v_1,\ldots,v_k)$ with 
\[1=v_1\leq v_2\leq\cdots\leq v_k\leq\frac{p-1}{2}.\] 
% This implementation results in at most $k$ speed tuples per equivalence class, 
% so there is a possibility to improve the run time by a factor of $k$. 
% However, we are unable to find a way to generate a single element from each equivalence class.  

We should stress that this reduction is essential to obtaining our results. 
There are approximately $p^k$ tuples in the $(k,p,1)$-ansatz,
but with this reduction we only need to check a much smaller number of approximately
\[\binom{p/2}{k-1}\approx\frac{p^{k-1}}{2^{k-1}(k-1)!}.\]
Compared to \cite{Rosenfeld} and \cite{Trakulthongchai}, 
we reduce the number of tuples to check by a factor of approximately $\frac{p}{2k}$. 

\subsection{Verification algorithms}
\label{subsec:algo}
Fix $k\in\{10,11,12\}$ and a prime $p$. Our goal is to verify that $J(k,p)=\emptyset$
in order to apply Corollary \ref{cor:evenpropdivp}. To do this, we use a verification algorithm that consists of two phases:
computing $I(k, p, 1)$ and repeated lifting.

The first phase is done by choosing representative speed tuples as described in the previous subsection. 
We use lifting diagram to explain the second phase. In a lifting diagram, the first set $S_1$ is always the set of representative tuples of the precomputed $I(k,p,1)$.

Two kinds of arrows correspond to two actions from Section \ref{sec:sieves}:
\begin{itemize}
    \item $S_n\xrightarrow{\times c}S_{n+1}$ indicates lifting from $S_n$ to $S_{n+1}$ with constant $c$.
    \item $S_n\xrightarrow{\div d}S_{n+1}$ indicates projecting with $\pi_{dp\rightarrow p}$.
\end{itemize}

Since we only perform these two actions, by Proposition~\ref{prop:liftproj}, if $S_n = \emptyset$ for some $n$, then $J(k, p)$ is necessarily empty.

For $k=11$, since $k+1=12$ only has small prime factors, namely $2$ and $3$,
it is convenient to lift through its prime factors (and some extra sieves). Hence, the lifting diagram is
\[S_1\xrightarrow{\times 2}S_2\xrightarrow{\times 2}S_3\xrightarrow{\times 2}S_4\xrightarrow{\times 2}S_5\xrightarrow{\times 3}S_6\xrightarrow{\times 3}S_7.\]
If $S_7$ is empty, we conclude that $\pi_p S_7$, which contains all the representatives of $J(k, p)$, is empty.

For $k\in\{10,12\}$, $k+1$ is prime, and here we use $p>k^2+k$. 
In order to avoid the lift with $c=k+1$, we lower the goal of the lifting phase from showing that the final set is empty, 
to only showing that the set contains only speed tuples equivalent to $\onetok$.
The lifting diagram for $k = 10$ and $k=12$ is 
\[ S_1\xrightarrow{\times 2}S_2\xrightarrow{\times 2}S_3\xrightarrow{\times 2}S_4\xrightarrow{\div 8}S_5 \] 
which in particular does not contain a $c=11$ or $c=13$ lift.

Proposition~\ref{prop:liftproj} asserts that $S_5$ contains every representative of $J(k, p)$. But in particular, as $k+1$ is prime and $p>k^2+k$,\footnote{
  In fact, we can use some primes $p<k^2+k$ if $\rbf_k(\frac{1}{k+1}\Z)\subseteq\rbf_k(\frac{1}{p}\Z)$.
  We have checked, by imitating the proof of Lemma~\ref{lem:resolvep}, 
  that this inclusion holds for $p \in \{103, 107, 109\}$ when $k=10$, and $p \in \{149,151\}$ when $k=12$,
  and for these $k,p$, $J(k,p)=\emptyset$.
} by Proposition \ref{prop:corlarge}, the set $J(k,p)$ cannot contain $\onetok$, nor its equivalents, by Proposition \ref{prop:equivproper}.
Therefore, if by computer verification we found that $S_5$ consists only of equivalents of $\onetok$, we have $J(k,p)=\emptyset$.

\subsection{Further particulars}
\label{subsec:particulars}
We implement the algorithms of Subsection \ref{subsec:algo} in \texttt{C++}.
Compared to the original implementations of \cite{rosenfeldcode} and \cite{tanupatcode},
we improve practical performance by refactoring to support parallel execution and by using more memory-efficient data structures. The
development of our code was assisted by AI-based tools, but the process was closely supervised
and the code was manually verified. 
We also tested our implementation against that of~\cite{tanupatcode} on a number of smaller cases.

In computing $S_1$, we modified the depth-first-search technique developed by \cite{rosenfeldcode} to account for the equivalence relation
mentioned in Subsection~\ref{subsec:equiv}.

These refinements dramatically improve runtime. We verified the case of $k=8$ in 2 seconds (15 minutes using \cite{tanupatcode}), 
$k=9$ in 41 seconds (23 hours using \cite{tanupatcode}), and $k=10$ in 45 minutes on a 10-core Apple M4 processor
(excluding compilation time). 
The proofs of $k=11$ and $k=12$ were run in multiple batches, with multiple versions, 
and on multiple machines, making the exact time taken incomparable to other cases. 

A heuristic estimate from the smaller cases is that the time required to verify the emptiness of $J(k, p)$ grows in proportion to $\frac{p^{(k+1)/2}}{k2^k}$, 
which improves over \cite{rosenfeld2025lonely2}\footnote{Rosenfeld \cite{rosenfeld2025lonely2} reported heuristic asymptotic runtimes for $k=8,9$. 
For $k=8$, our runtime is $\Theta(p^{4.5})$ compared to his $\Theta(p^{5.7})$, and for $k=9$, our runtime is $\Theta(p^5)$ compared to his $\Theta(p^{6.5})$, 
ignoring the constant term depending only on $k$. An improvement by a factor of $p$ is due to the reduction modulo $p$ in Subsection \ref{subsec:equiv}, 
and the remaining likely arises from sieving.}. Extrapolating this model suggests that completing the proof for $k=11$ will take approximately 40 hours, 
and for $k=12$ approximately 40 days, on the same 10-core machine.

The source of the computation, along with the log files and detailed descriptions, 
can be found in the first author's GitHub repository \cite{mycode}.

\section{Proof of Theorem \ref{thm:main}}
\label{sec:result}
Computer verification shows that, for each $k\in\{10,11,12\}$, the set $P_k$ in Table~\ref{tab:primes} satisfies
$J(k,p) = \emptyset$ for every $p\in P_k$. Moreover, for each $k$, we have that $\prod_{p \in P_k} p\geq B_k$. 
Since $LRC(9)$ holds \cite{Trakulthongchai},
we apply Proposition \ref{prop:verify} thrice and conclude that $LRC(k)$ holds for $k\in\{10,11,12\}$.

\begin{table}[ht]
\centering
\begin{tabular}{|c|l|c|c|}
\hline
$k$ & \multicolumn{1}{c|}{$P_k$}  
    & $\ln\displaystyle\prod_{p\in P_k}p$ & 
    \begin{tabular} [c]{@{}c@{}}$\ln B_k$\\ (Lemma~\ref{lem:proddivB}) \end{tabular} \\ \hline
10  &
    $\mathbb{P}\cap[127,467]$
    & $>342$  & $<338$ \\ \hline
11  & $\left(\mathbb{P}\cap[131,577]\right)\cup\{23\}$
    & $>435$ & $<435$ \\ \hline
12  & $\left(\mathbb{P}\cap[167,733]\right)\setminus\{173\}$ & $>547$ & $<546$ \\ \hline
\end{tabular}
\vspace{5pt}
\caption{Sets of primes used in the verification of $LRC(k)$ for $k\in\{10,11,12\}$, where $\mathbb{P}$ denotes the set of primes.}
\label{tab:primes}
\end{table}

\section{Concluding remarks}
\label{sec:conclude}
\subsection{Non-tight speed tuples in an ansatz}
For a speed tuple $\vbf$, define the maximum loneliness function \[\ml(\vbf)=\sup_{t\in\R}\min_{i\in\{1,\ldots,k\}}\norm{tv_i}\] 
(as $\min_{i\in\{1, \ldots, k\}} \norm{tv_i}$ is continuous on $\R/\Z$, $\sup$ can be replaced by $\max$).
Then $LRC(k)$ corresponds to the inequality $\ml(\vbf)\geq\frac{1}{k+1}$ for every $\vbf\in\Z_{>0}^k$. 
Goddyn and Wong \cite{goddyn2006tight} studied \emph{tight} speed tuples, namely speed tuples where the equality is attained. 
Indeed, tightness is precisely the condition that necessitates a sieve with $l=k+1$, as described in Remark \ref{rem:tightset}. 
Luckily for our verification, tight speed tuples are vanishingly rare, 
the only known (and likely only existing) tight tuples for $k\in\{10,11,12\}$ being $\onetok$, up to scaling and permutation.

Nonetheless, the rarity of tight speed tuples does not explain why non-tight speed tuples can be sieved out so easily 
(with only three $c=2$ lifts for our prime $k+1$ case); we have the following heuristic explanation.

An easy continuity argument shows that if $\vbf$ is not tight, 
then $LRC(k)$ implies that there are $T$ and $\varepsilon>0$ such that every $t\in(T,T+\varepsilon)$ is a witness time for $\vbf$,
so there is $L$ (depending on $\vbf$) such that for any $l\geq L$ there is a witness time in $\frac{1}{lp}\Z$. 
In fact, it turns out that for $p$ sufficiently large, we can choose $l=1$ to eliminate every equivalence class containing no tight representative from an ansatz.

\begin{proposition}
\label{prop:lrcimpliestight}
  Assume that $LRC(k-1)$ holds. For a fixed $k$, $LRC(k)$ holds if and only if the following holds: There exists a constant $p_k$ such that, for each prime $p\geq p_k$, 
  every speed tuple $\vbf\in\Z_{p,1}^k$ not equivalent to a tight speed tuple $\ubf\in\Z_{p,1}^k$ has a witness time in $\frac{1}{p}\Z$.
\end{proposition}
\begin{proof}[Proof sketch (from ChatGPT-5.6 Pro)]
    Assume $k\geq2$ and $LRC(k)$. There is a uniform $\varepsilon_k>0$ such that $\ml(\vbf)\geq\frac{1}{k+1}+\varepsilon_k$ for every non-tight $\vbf$ \cite[Theorem 1.4]{giri_structure_2026}. For any $\vbf\in\Z_{p,1}^k$, by a standard pigeonhole argument we can find $\ubf\in\Z_{p,1}^k$ equivalent to $\vbf$ with $\max(\ubf)\leq O(p^{1-1/k})$. If $\ubf$ is not tight, there is a witness time interval of length $\Omega(\frac{2\varepsilon_k}{p^{1-1/k}})$ which is certainly larger than $\frac{1}{p}$ when $p$ is large, so there is a witness time in $\frac{1}{p}\Z_p$. By equivalence $\vbf$ also has a witness time in $\frac{1}{p}\Z_p$.
\end{proof}
The speed tuples behave increasingly nicely as $p$ gets larger because 
the $\norm{\cdot}_\infty$ norm of the smallest tuple in each equivalence class as a proportion of $p$ decreases. Still, the constant $p_k$ given by the proof here is enormous. 
From \cite[Theorem 3.1]{kravitz2021barely}, we have $\frac{s}{ks+1}\in\mathcal{\Tilde{S}}(k)$ for every $s\in\Z_{>0}$, hence
\[\varepsilon_k\leq\frac{2}{2k+1}-\frac{1}{k+1}<\frac{1}{2k^2}.\] 
By this method we must then choose $p_k\approx(2\varepsilon_k)^{-k}>k^{2k}$.\footnote{The similarity between this linear-exponential form and the one in \cite[Theorem A]{Malikiosis} is no coincidence. Both arise from a gap of order $k^{-2}$ and a pigeonhole argument in $k+O(1)$ dimension.} 
This is much larger than what the empirical evidence suggests (which is in the hundreds for our $k$); it would be interesting to see how this gap can be closed.

\subsection{Next cases}
Proposition \ref{prop:polynomial} brings into the picture a finite field reduction of some cases of the conjecture. 
The algebraic method there may be generalizable to verify Conjecture \ref{lonelyconjred} for other classes of speed tuples. 
This method is particularly useful when the speeds are small, so improving this could be complementary to the analytic methods that are powerful when the speeds are large.

Looking ahead, from the same heuristic estimate as in Subsection \ref{subsec:particulars}, our algorithm will take $15$ years on our original machine to verify $LRC(13)$. 
The primary bottleneck here is the efficient computation of $I(k, p, 1)$. 
Progress then likely requires a better understanding of speed tuples that do not have a witness time in an ansatz, 
which in turn could yield stronger pruning conditions and make the initial sieve more tractable. 

\section*{Acknowledgements}
The authors are most thankful to Noah Kravitz for his insightful feedback and suggestions, 
especially on the exposition of the paper. We thank Jörg Wills, Jannis Dietrich, the GitHub user \texttt{Stayfoool}, and the referees for helpful comments.
The second author acknowledges the support from the King's Scholarship (Thailand).

\bibliographystyle{amsplain}
\bibliography{ref}

\end{document}